\documentclass[15pt]{article}
\usepackage{amsmath, amsthm, amsfonts, amssymb, enumerate}
\usepackage[pdftex]{graphicx}
\usepackage{natbib}

\usepackage[a4paper,bindingoffset=0.2in,left=1in,right=1in,top=1in,bottom=1in,footskip=.25in]{geometry}
\usepackage[pdfencoding=auto,psdextra,colorlinks,citecolor=blue]{hyperref}
                  
\usepackage{def_oli}

\newcommand{\R}{\mathbb{R}}
\newcommand{\E}{\mathbb{E}}
\newcommand{\ddr}{\mathrm{d}}
\renewcommand{\P}{\mathbb{P}}
\newcommand{\edr}{\mathrm{e}}

\def\E{\mathbb{E}}
\DeclareMathOperator{\Var}{Var}
\DeclareMathOperator{\Beta}{Beta}
\DeclareMathOperator{\Dir}{Dir}
\DeclareMathOperator{\Bern}{Bern}

\def\Rcode{The \textsf{R} code for the plots presented in this note and for a function deriving the optimal proxy variance in terms of $\alpha$ and $\beta$ is available at \href{http://www.julyanarbel.com/software}{\textsf{http://www.julyanarbel.com/software}}.}

\theoremstyle{plain}
\newtheorem{thm}{{Theorem}}
\newtheorem{defi}{{Definition}}
\newtheorem{cor}{{Corollary}}
\newtheorem{rem}{{Remark}}

\title{On the sub-Gaussianity of the Beta and Dirichlet distributions}
\author{Olivier Marchal$^1$ and Julyan Arbel$^2$\vspace{5mm}\\
$^1$  Universit\'{e} de Lyon, CNRS UMR 5208, Universit\'{e} Jean Monnet,
\\
Institut Camille Jordan, France\vspace{5mm}\\
$^2$  Inria Grenoble Rh\^one-Alpes,\\
Laboratoire Jean Kuntzmann, Universit\'e Grenoble Alpes, France
}

\begin{document}

\maketitle

\begin{abstract}
We obtain the optimal proxy variance for the sub-Gaussianity of Beta distribution, thus proving upper bounds recently conjectured by Elder (2016). We provide different proof techniques for the symmetrical (around its mean) case and the non-symmetrical case. The technique in the latter case relies on studying the ordinary differential equation satisfied by the Beta moment-generating function known as the  confluent hypergeometric function. As a consequence, we derive the optimal proxy variance for the Dirichlet distribution, which is apparently a novel result. We also provide a new proof of the optimal proxy variance for the Bernoulli distribution, and discuss in this context the proxy variance relation to log-Sobolev inequalities and transport inequalities. 
\end{abstract}

\section{Introduction}

The sub-Gaussian property \citep{buldygin1980sub,buldygin2000metric,pisier2016subgaussian} and related concentration inequalities \citep{boucheron2013concentration,raginsky2013concentration} have attracted a lot of attention in the last couple of decades due to their applications in various areas such as pure mathematics, physics, information theory and computer sciences. Recent interest focused on deriving the optimal proxy variance for discrete random variables like the Bernoulli distribution \citep{buldygin2000binary,kearns1998large,berend2013concentration} and the missing mass \citep{mcallester2000convergence,mcallester2003concentration, berend2013concentration,ben2017concentration}.
Our focus is instead on two continuous random variables, the Beta and Dirichlet distributions, for which the optimal proxy variance was not known to the best of our knowledge. Some upper bounds were recently conjectured by \cite{elder2016bayesian} that we prove in the present article by providing the optimal proxy variance for both Beta and Dirichlet distributions. Similar concentration properties of the Beta distribution have been recently used in many contexts including  Bayesian adaptive data analysis \citep{elder2016bayesian},  Bayesian nonparametrics \citep{castillo2016polya} and spectral properties of random matrices \citep{perry2016statistical}.


We start by reminding the definition of sub-Gaussian property for random variables:

\begin{defi}[Sub-Gaussian variables]\label{Def1}
A random variable $X$ with finite mean $\mu=\E[X]$ is {sub-Gaussian} if there is a positive number $\sigma$ such that:
\begin{align}\label{eq:def}
\E[\exp(\lambda (X-\mu))]\le\exp\left(\frac{\lambda^2\sigma^2}{2}\right)\,\,\text{for all } \lambda\in\R.
\end{align}
Such a constant $\sigma^2$ is called a {proxy variance} (or sub-Gaussian norm), and we say that {$X$ is $\sigma^2$-sub-Gaussian}. If $X$ is sub-Gaussian, one is usually interested in {the optimal proxy variance}:
\beqq \sigma_{\text{opt}}^2(X)=\min\{\sigma^2\geq 0\text{ such that } X \text{ is } \sigma^2\text{-sub-Gaussian}\}.\eeqq
Note that the variance always gives a lower bound on the optimal proxy variance: $\text{Var}[X]\leq \sigma_{\text{opt}}^2(X)$. In particular, when $\sigma_{\text{opt}}^2(X)=\text{Var}[X]$, $X$ is said to be {strictly sub-Gaussian}.
\end{defi}
Every compactly supported distribution, as is the Beta$(\alpha,\beta)$ distribution, is sub-Gaussian. This can be seen by Hoeffding's classic inequality: any random variable $X$ supported on $[0,1]$ with mean $\mu$ satisfies
\begin{equation*}
	\forall \lambda\in\R,\quad \E\left[\edr^{\lambda(X-\mu)}\right]\leq \edr^{\frac{\lambda^2}{8}},
\end{equation*}
thus exhibiting $\frac{1}{4}$ as an upper bound to the proxy variance. This bound can be improved by taking into account the location of the mean $\mu$ within the interval $[0,1]$. An early step in this direction is the second inequality in \cite{hoeffding1963probability} paper, indexed (2.2). It states that if $\mu<1/2$, then for any positive $\epsilon$, $\P(X-\mu>\epsilon)\leq \edr^{-\epsilon^2g(\mu)}$, where 
\begin{equation}\label{eq:g}
g(\mu) = \frac{1}{1-2\mu}\ln\frac{1-\mu}{\mu}
\end{equation}
thus indicating that $X$ has a right tail lighter than a Gaussian tail of variance $\frac{1}{2g(\mu)}$. Hoeffding's result was strengthened by \cite{kearns1998large} to comply with Definition~\ref{Def1} of sub-Gaussianity\footnote{Note indeed that Equation~\eqref{eq:def}, together with Markov inequality, imply $\P(X-\mu>\epsilon)  \leq \edr^{-\frac{\epsilon^2}{2\sigma^2}}$.} as follows
\begin{align}\label{eq:kearns}
\E[\exp(\lambda (X-\mu))]\le\exp\left(\frac{\lambda^2}{4g(\mu)}\right)\,\,\text{for all } \lambda\in\R,
\end{align}
thus indicating that  $\frac{1}{2g(\mu)}$ is a distribution-sensitive proxy variance for any $[0,1]$-supported random variable with mean $\mu$ \citep[see also][for a detailed proof of this result]{berend2013concentration}. If this is the optimal proxy variance for the Bernoulli distribution \citep[see Theorem 2.1 and Theorem 3.1 of][]{buldygin2000binary}, it is clear from our result that it does not hold true for the Beta distribution. 
However, fixing $\frac{\alpha}{\alpha+\beta} = \mu$ and letting $\alpha\to0$, $\beta\to0$, the $\Beta(\alpha,\beta)$ distribution concentrates to the $\Bern(\mu)$ distribution, and we show that we recover the optimal proxy variance for the Bernoulli distribution (Theorem~\ref{thm:3}).

An interesting common feature between optimal proxy variances for the Bernoulli distribution:  $\frac{1}{2g(\mu)}$, and that of the Beta distribution derived later on, is that they deteriorate in  a similar fashion as the mean $\mu$ goes to 0 or 1, see for instance the left panel of Figure~\ref{fig:proxy}. We briefly present here classical proof techniques for sub-Gaussianity hinging on certain tools from functional analysis. We show how they apply in the Bernoulli setting, and let as an interesting open problem how our proof in the Beta distribution setting could be supplemented by these same functional analysis tools. 

Essentially two (related) functional inequalities allow one to derive a sub-Gaussian property: log-Sobolev inequalities, which date back to \cite{gross1975logarithmic}, and transport inequalities. The relation with the former inequalities is called Herbst's argument. It states that if a probability measure satisfies a  log-Sobolev inequality with some constant, then  it is sub-Gaussian with the same constant as a proxy variance\footnote{The implied predicate is actually stronger than sub-Gaussianity, but it is not useful for our purposes.}  \citep[see for instance][Section 2.3 and Proposition 2.3]{ledoux1999concentration}. The optimal constant in the log-Sobolev inequality satisfied by the Bernoulli distribution also produces its optimal proxy variance \citep[][Corollary 5.9]{ledoux1999concentration}.

The relation with transport inequalities is usually referred to as Marton's argument \citep[see for instance][Section 3.4]{raginsky2013concentration}. 
Define the Wasserstein distance between two probability measures $P$ and $Q$ on a space $\mathcal{X}$ by 
\begin{equation*}
W(P,Q) = \inf_{\pi\in\Pi(P,Q)}\int_{\mathcal{X}\times \mathcal{X}}d(x,y)\pi(\ddr x, \ddr y),
\end{equation*}
where $\Pi(P,Q)$ is the set of probability measures on $\mathcal{X}\times \mathcal{X}$ with fixed marginal distributions respectively $P$ and $Q$. The  Wasserstein distance  depends on some choice of a distance $d$ on $\mathcal{X}$. A probability measure $P$ is said to satisfy a transport inequality with constant $c$, if for any probability measure $Q$ dominated by $P$,
\begin{equation}\label{eq:transport}
W(P,Q) \leq \sqrt{2c D(Q||P)},
\end{equation}
where $D(Q||P)$ is the entropy, or Kullback--Leibler divergence, between $P$ and $Q$. The transport inequality~\eqref{eq:transport} is denoted by $\mathsf{T}(c)$. 

\cite{bobkov1999exponential} proved that $\mathsf{T}(c)$ implies $c$-sub-Gaussianity. See also Proposition 3.6 and Theorem 3.4.4 of \cite{raginsky2013concentration} for general results. 
Further developments in the discrete $\mathcal{X}$ setting are interesting for our purposes. Equip a discrete space $\mathcal{X}$ with the Hamming metric, $d(x,y) = \mathbb{1}_{\{x\neq y\}}$. The induced Wasserstein distance then reduces to the total variation distance, $W(P,Q) = \Vert P-Q\Vert_{\text{TV}}$. In that setting, \cite{ordentlich2005distribution} proved the distribution-sensitive transport inequality:
\begin{equation}\label{eq:discrete_transport}
\Vert P-Q\Vert_{\text{TV}} \leq \sqrt{\frac{1}{g(\mu_P)}D(Q||P)},
\end{equation}
where the function $g$ is defined in Equation~\eqref{eq:g} and the coefficient $\mu_P$ is called the balance coefficient of $P$, and is defined by $\mu_P = \underset{A\subset \mathcal{X}}\max\min \{P(A), 1-P(A)\}$. In particular, the Bernoulli balance coefficient is easily shown to coincide with its mean. Hence, applying the result  of \cite{bobkov1999exponential} to the $\mathsf{T}\left(\frac{1}{2g(\mu_P)}\right)$ transport inequality~\eqref{eq:discrete_transport} yields a distribution-sensitive proxy variance of $\frac{1}{2g(\mu)}$ for the Bernoulli with mean $\mu$. It is optimal, see for instance Theorem 3.4.6 of \cite{raginsky2013concentration}.  This viewpoint highlights the key role played by the balance coefficient in the non-uniformity of the optimal proxy variance for discrete distributions such as the Bernoulli. However, it is not clear how this argument would carry over to non discrete distributions such as the Beta distribution for explaining similar sensitivity to the mean. However, to quote \cite{raginsky2013concentration}, the general approach may not produce optimal concentration estimates, that often require case-by-case treatments. This is the route followed in this note for the Beta distribution.


The outline of the note is as follows. We introduce the Beta distribution and state the main result (Theorem~\ref{thm:1}) in Section~\ref{sec:notations_results}. We then prove our result 
 depending on whether $\alpha=\beta$ (Section~\ref{sec:equal}) or $\alpha\neq\beta$ (Section~\ref{sec:distinct}). In the first case, the proof is elementary and based on comparing the coefficients of the entire series representations of the functions of both sides of inequality~\eqref{eq:def}. However, it does not directly carry over to the second case, whose proof requires some finer analysis tool: the study of the ordinary differential equation (ODE) satisfied by the confluent hypergeometric function $_1F_1$. 
Although the second proof also covers the case $\alpha=\beta$ upon slight modifications, the independent proof for the symmetric case is kept owing to its simplicity. As a by-product, we derive the optimal proxy variance for the Bernoulli and the Dirichlet distributions in Section~\ref{sec:dir}. 
\Rcode

\section{Optimal proxy variance for the Beta distribution}
\subsection{Notations and main result\label{sec:notations_results}}

The $\Beta(\alpha,\beta)$ distribution, with $\alpha,\beta>0$, is characterized by a density on the segment $[0,1]$ given by:
\begin{align*}
f(x) = \frac{1}{B(\alpha,\beta)} x^{\alpha-1}(1-x)^{\beta-1},
\end{align*}
where $B(\alpha,\beta) = \int_0^\infty x^{\alpha-1}(1-x)^{\beta-1} \ddr x = \frac{\Gamma(\alpha)\Gamma(\beta)}{\Gamma(\alpha+\beta)}$ is the Beta function. The moment-generating function of a $\text{Beta}(\alpha,\beta)$ distribution is given by a confluent hypergeometric function (also known as Kummer's function):
\begin{align}\label{eq:HyperGeo}
\E[\exp(\lambda X)]\,= \,_1F_1(\alpha;\alpha+\beta;\lambda)=\sum_{j=0}^\infty\frac{\Gamma(\alpha+j)\Gamma(\alpha+\beta)}{(j!) \Gamma(\alpha)\Gamma(\alpha+\beta+j)}\lambda^j.
\end{align}
This is equivalent to say that the $j^{\text{th}}$ raw moment of a $\Beta(\alpha,\beta)$ random variable $X$ is given by:
\begin{align}\label{eq:RawMoments}
 \E[X^j] = \frac{(\alpha)_j}{(\alpha+\beta)_j},
\end{align}
where $(x)_j=x(x+1)\cdots(x+j-1)=\frac{\Gamma(x+j)}{\Gamma(x)}$ is the \textit{Pochhammer symbol}, also known in the literature as a \textit{rising factorial}. In particular, the mean and variance are given by:
\beqq \E[X]=\frac\alpha{\alpha+\beta},\quad\Var[X]=\frac{\alpha\beta}{(\alpha+\beta)^2(\alpha+\beta+1)}.\eeqq

The Beta distribution is ubiquitous in statistics. It plays a central role in the binomial model in Bayesian statistics where it is a conjugate prior distribution (the associated posterior distribution is also Beta): if $X\sim\text{Binomial}(\theta,N)$ and $\theta\sim\Beta(\alpha,\beta)$, then $\theta|X\sim\Beta(\alpha+X, \beta+N-X)$. It is also key to Bayesian nonparametrics where it embodies, among others, the distribution of the breaks in the stick-breaking representation of the Dirichlet process and the Pitman--Yor process; marginal distributions of Polya trees \citep{castillo2016polya}; 
the posterior distribution of discovery probabilities under a Bayesian nonparametrics model \citep{arbel2015discovery}. Our main result opens new research avenues for instance about asymptotic (frequentist)  assessments of these procedures.\bigskip

Our main result regarding the Beta distribution is the following:
\begin{thm}[Optimal proxy variance for the Beta distribution]\label{thm:1}
For any $\alpha,\beta>0$, the Beta distribution $\Beta(\alpha,\beta)$ is $\sigma_{\text{opt}}^2(\alpha,\beta)$-sub-Gaussian with optimal proxy variance $\sigma_{\text{opt}}^2(\alpha,\beta)$ given by:
\begin{equation}\label{eq:def_x0}
\left\{\begin{array}{ll}
        \sigma_{\text{opt}}^2(\alpha,\beta)=\frac{\alpha}{(\alpha+\beta)x_0}\left( \frac{_1F_1(\alpha+1;\alpha+\beta+1;x_0)}{_1F_1(\alpha;\alpha+\beta;x_0)}  -1\right)\cr
\quad\quad \text{ where }x_0 \text{ is the unique solution of the equation}\cr
		\ln(_1F_1(\alpha;\alpha+\beta;x_0))=\frac{\alpha x_0}{2(\alpha+\beta)}\left(1+\frac{_1F_1(\alpha+1;\alpha+\beta+1;x_0)}{_1F_1(\alpha;\alpha+\beta;x_0)}\right).
    \end{array}
\right.  
\end{equation}
A simple and explicit upper bound to $\sigma_{\text{opt}}^2(\alpha,\beta)$
is given by $\sigma_0^2(\alpha,\beta) = \frac{1}{4(\alpha+\beta+1)}$:\\
- for $\alpha\neq \beta$ we have $\text{\normalfont{Var}}[\Beta(\alpha,\beta)]<\sigma_{\text{opt}}^2(\alpha,\beta)<\frac{1}{4(\alpha+\beta+1)}$\\
- for $\alpha=\beta$ we have $\text{\normalfont{Var}}[\Beta(\alpha,\alpha)] = \sigma_{\text{opt}}^2(\alpha,\alpha)=\frac{1}{4(2\alpha+1)}$. 
\end{thm}
Equation~\eqref{eq:def_x0} defining $x_0$  is a transcendental equation, the solution of which is not available in closed form. However, it is simple to evaluate numerically. 
The values of the variance, optimal proxy variance and its simple upper bound are illustrated on Figure~\ref{fig:proxy}. Note that for a fixed value of the sum of the parameters, $\alpha+\beta = S$, the optimal proxy variance deteriorates when $\alpha$, or equivalently $\beta$,  gets close to 0 or to $S$. This is reminiscent of the Bernoulli optimal proxy variance behavior which deteriorates when the success probability moves away from $\frac{1}{2}$ \citep{buldygin2000binary}.
\begin{figure}[!b]
\includegraphics[trim={1cm 0.5cm 0.5cm 0.5cm},clip,width=.33\textwidth]{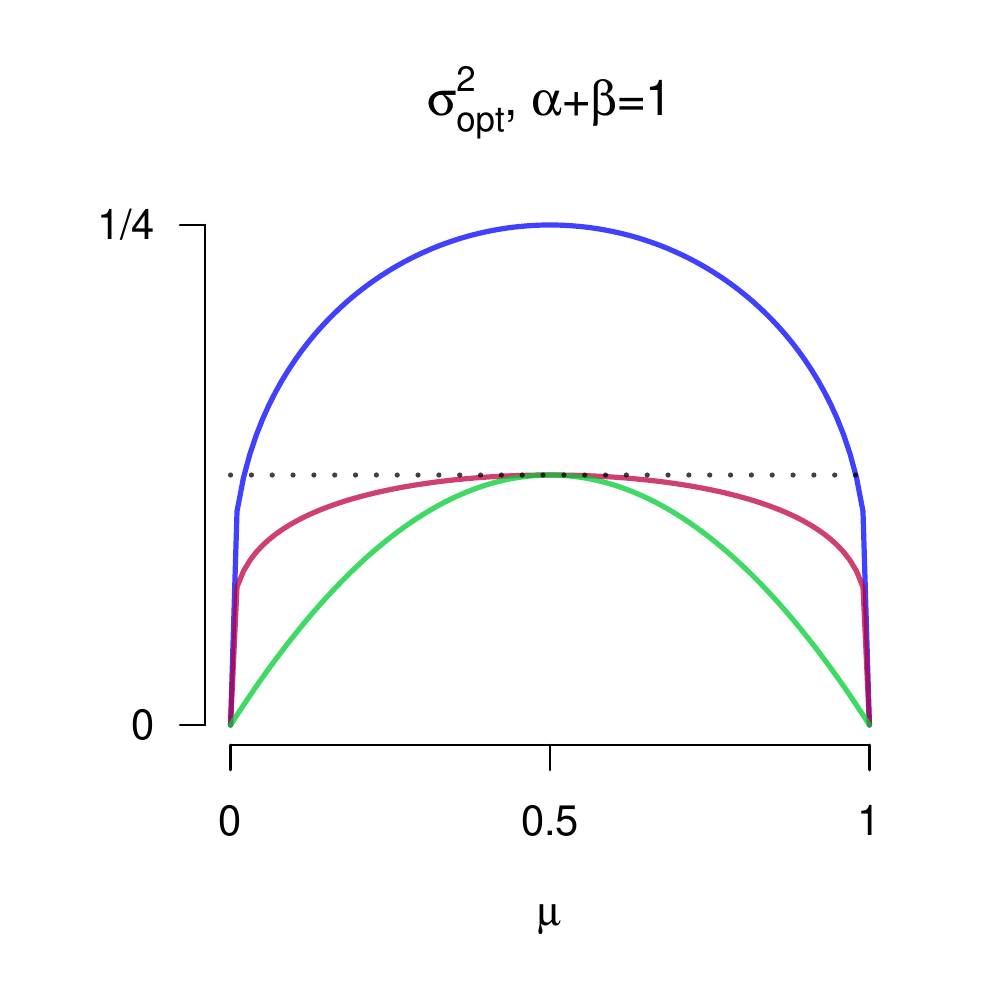} 
\includegraphics[trim={1cm 0.5cm 0.5cm 0.5cm},clip,width=.33\textwidth]{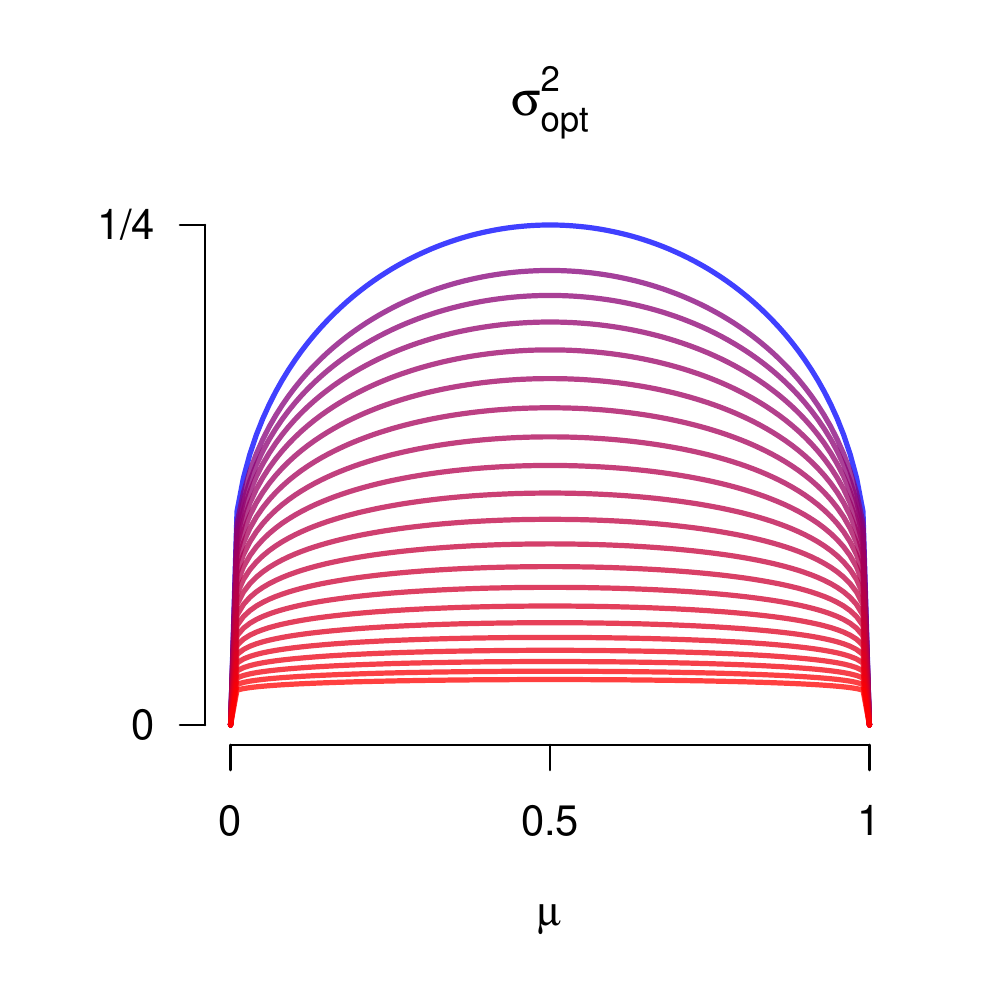} 
 \includegraphics[trim={6.5cm 1cm 5.1cm 2.7cm},clip,width=.3\textwidth]{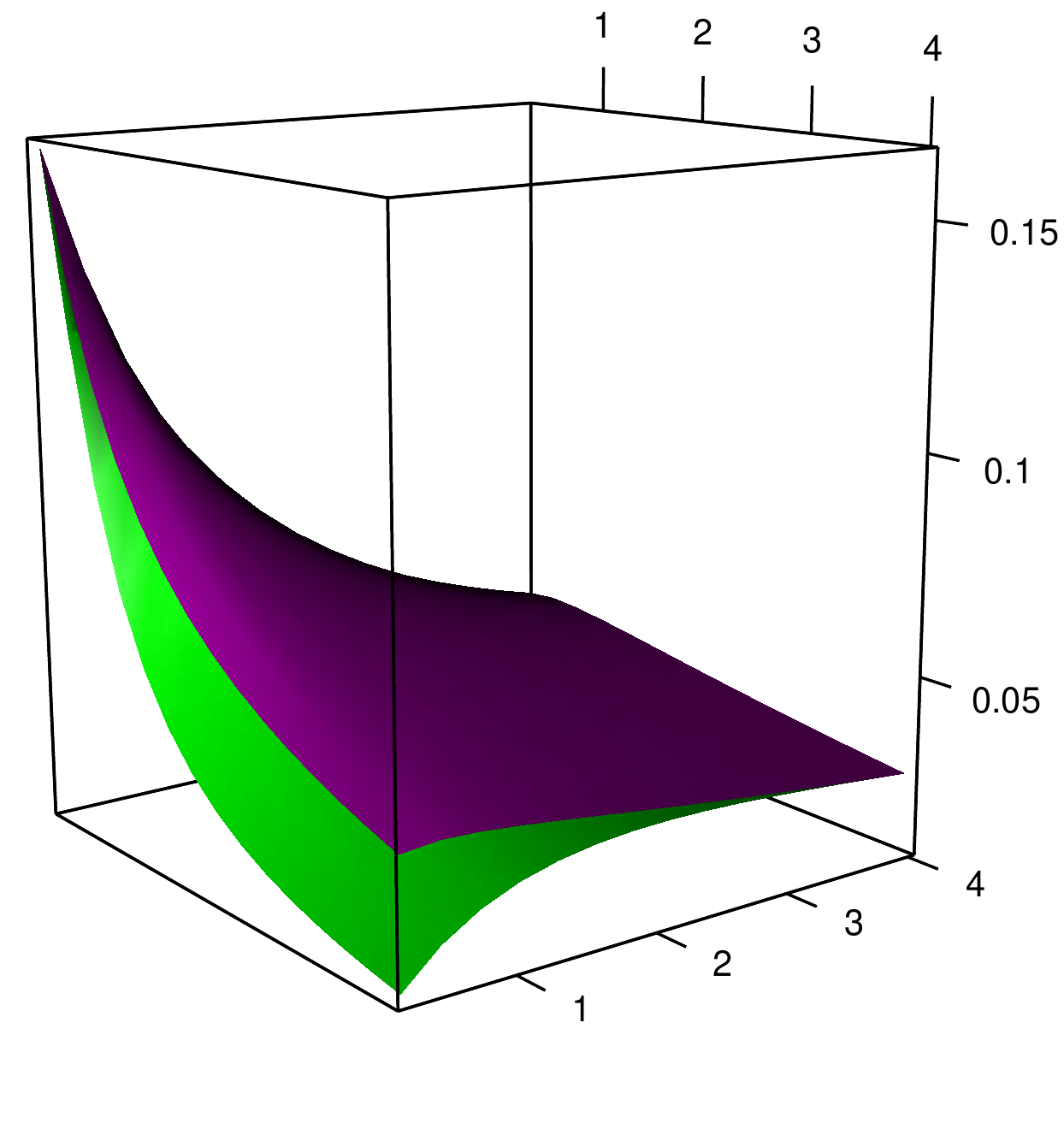}
\caption{\label{fig:proxy}
\textit{Left}: curves of $\text{\normalfont{Var}}[\Beta(\alpha,\beta)]$ (green), $\sigma_{\text{opt}}^2(\alpha,\beta)$ (purple) and $\frac{1}{4(\alpha+\beta+1)}$ (dotted black) for the $\Beta(\alpha,\beta)$ distribution with $\alpha+\beta$ set to 1, $\sigma_{\text{opt}}^2(\mu)$ for the $\Bern(\mu)$ distribution (blue); varying mean $\mu$ on the $x$-axis. 
\textit{Center}: curves of  $\sigma_{\text{opt}}^2(\mu)$ for the $\Bern(\mu)$ distribution  (blue), and of $\sigma_{\text{opt}}^2(\alpha,\beta)$ for the $\Beta(\alpha,\beta)$ distribution with $\alpha+\beta$ varying on a log scale from 0.1 (purple) to 10 (red); varying mean $\mu$ on the $x$-axis. 
\textit{Right}: surfaces of $\text{\normalfont{Var}}[\Beta(\alpha,\beta)]$ (green) and $\sigma_{\text{opt}}^2(\alpha,\beta)$ (purple), for values of $\alpha$ and $\beta$ varying in $[0.2, 4]$. 
}
\end{figure}
The intuition of the proof can be seen from Figure~\ref{fig:proof} (Section \ref{FIG2}) where we represent the difference $\lambda\mapsto\exp\left(\mathbb{E}[X]\lambda+\frac{\sigma^2}{2}\lambda^2\right)-\E[\exp(\lambda X)]$ for various values of $\sigma^2$. The main argument is that the optimal proxy variance is obtained for the curve (in magenta) whose positive local minimum equals zero, thus leading to the system of equations of Theorem~\ref{thm:1}. \bigskip
\begin{cor}\label{cor}
The Beta distribution $\Beta(\alpha,\beta)$ is strictly sub-Gaussian if and only if $\alpha=\beta$.
\end{cor}
As a direct consequence, we obtain the strict sub-Gaussianity of the uniform, the arc-sine and the Wigner semicircle distributions, as special cases up to a trivial rescaling of the Beta$(\alpha,\alpha)$ distribution respectively with $\alpha$ equal to $1$, $\frac{1}{2}$ and $\frac{3}{2}$.

\subsection{The $\text{Beta}(\alpha,\alpha)$ distribution is strictly sub-Gaussian\label{sec:equal}}
Let $\sigma_0^2(\alpha)=\text{Var}[\Beta(\alpha,\alpha)]=\frac{1}{4(2\alpha+1)}$. Since a random variable $X\sim\Beta(\alpha,\alpha)$ is symmetric around $\frac{1}{2}$, only its even centered moments are non-zero. The reason why $\E[\exp(\lambda(X-\E[X]))]\le\exp\big(\frac{\lambda^2\sigma_0^2(\alpha)}{2}\big)$ is because the coefficients of the series expansions at $\lambda=0$ of each side: 
\begin{align}\label{eq:AlphaAlphaCase}
\E[\exp(\lambda(X-\E[X]))] &= \sum_{j=0}^{\infty} \E\left[(X-1/2)^{2j}\right]\frac{\lambda^{2j}}{(2j)!},\\
\exp\left(\frac{\lambda^2\sigma_0^2(\alpha)}{2}\right)&=\sum_{j=0}^{\infty} \frac{\sigma^{2j}}{2^{j}}\frac{\lambda^{2j}}{j!}=\sum_{j=0}^{\infty} \frac{\lambda^{2j}}{2^{2j}2^j(2\alpha+1)^j (j!)},
\end{align}
satisfy the inequalities:
\begin{align}\label{eq:condition}
 \E\left[\left(X-\frac{1}{2}\right)^{2j}\right]\frac{1}{(2j)!} \leq \frac{\sigma_0^{2j}(\alpha)}{2^{j}}\frac{1}{j!}.  
\end{align}
Indeed, algebra yields:
\begin{align}\label{eq:centered_moments}
 \E\left[\left(X-\frac{1}{2}\right)^{2j}\right] = \frac{(2j)!}{2^{2j}j!}\frac{\Gamma(2\alpha )\Gamma(\alpha +j)}{\Gamma(\alpha )\Gamma(2(\alpha +j))}=\frac{(2j)!}{2^{2j}j!}\frac{(\alpha )_j}{(2\alpha )_{2j}}.
\end{align}
Combining the expression of the raw moments \eqref{eq:RawMoments} with the following inequality:
\begin{align}
\frac{(\alpha )_j}{(2\alpha )_{2j}} = \frac{1}{2^j\underset{l=1}{\overset{j}{\prod}}(2\alpha +2l-1)}\leq \frac{1}{(2(2\alpha +1))^j},
\end{align}
in \eqref{eq:AlphaAlphaCase} concludes the proof.

\begin{rem}
The non-symmetrical distribution with $\alpha\neq\beta$ has even centered moments whose expressions are not as simple as~\eqref{eq:centered_moments}. Moreover, it has obviously non-zero odd centered moments. For this last reason, the present proof does not carry over to the case $\alpha\neq \beta$.
\end{rem}

\subsection{Optimal proxy variance for the $\text{Beta}(\alpha,\beta)$ distribution\label{sec:distinct}}
\subsubsection{Connection with ordinary differential equations}
In this section, we assume that $X\sim\text{Beta}(\alpha,\beta)$ with $\beta\neq\alpha$. We denote $\sigma_0^2=\frac{1}{4(\alpha+\beta+1)}$ (we omit the dependence on $\alpha$ and $\beta$ for compactness) and define for all $t\in\R$: 
\begin{equation*}
	\sigma^2_t=\frac{(1-t)}{4(1+\alpha+\beta)}+t\frac{\alpha\beta}{(\alpha+\beta)^2(\alpha+\beta+1)}
	=\frac{1}{4(1+\alpha+\beta)}+\frac{(\beta-\alpha)^2}{4(\alpha+\beta)^2(1+\alpha+\beta)}t.
\end{equation*}
In other words, the decreasing function $t\mapsto \sigma^2_t$ maps the interval $[0,1]$ to the interval $[\sigma_1^2,\sigma_0^2]$ with $\sigma_1^2=\text{Var}[X]$. Then, we introduce the function $u_t$ defined by:
\beqq u_t(x)\overset{\text{def}}{=}\exp\left(\frac{\alpha}{\alpha+\beta} x+\frac{\sigma_t^2}{2}x^2\right)-\E[\exp(x X)] \,\,,\,\, \forall \, x\in \mathbb{R},\eeqq 
where $\sigma^2_t$-sub-Gaussianity amounts to non negativity of $u_t$ on $\R$. 
Since the confluent hypergeometric function $y: x\mapsto y(x) = \,_1F_1(\alpha, \alpha+\beta;x)$ satisfies the linear second order ordinary differential equation $xy''(x)+(\alpha+\beta-x)y'(x)-\alpha y(x)=0$, we obtain together with equation \eqref{eq:HyperGeo}  that $u_t$ is the unique solution of the Cauchy problem:
\beq \label{ODE3} \left\{\begin{array}{ll}
        xu_t''(x)+(\alpha+\beta-x)u_t'(x)-\alpha u_t(x)\cr
        \quad\quad\quad\quad\quad\quad\quad\quad=\frac{x}{16(\alpha+\beta)^4(1+\alpha+\beta)^2}\exp\left(\frac{\alpha}{\alpha+\beta} x+\frac{\sigma_t^2}{2}x^2\right)P_2(x;t),\cr
				u_t(0)=0 \text{ and } u_t'(0)=0,
    \end{array}
\right.
\eeq
where $P_2$ is a polynomial of degree $2$ in $x$:
\begin{multline*}
	P_2(x;t)=4(1-t)(\alpha^2-\beta^2)^2(1+\alpha+\beta)^2\\
	-4(\beta^2-\alpha^2)(1+\alpha+\beta)\left((\alpha+\beta)^2-t(\beta-\alpha)^2\right)x
	+\left((\alpha+\beta)^2-t(\beta-\alpha)^2\right)^2x^2.
\end{multline*}
For normalization purposes, we also define:
\begin{multline}
	\label{vt} v_t(x)=16(\alpha+\beta)^4(1+\alpha+\beta)^2 u_t(x)\exp\left(-\frac{\alpha }{\alpha+\beta}x-\frac{\sigma_t^2}{2}x^2\right),\\
=16(\alpha+\beta)^4(1+\alpha+\beta)^2\left(1-\E[\exp(x X)]\exp\left(-\frac{\alpha}{\alpha+\beta} x-\frac{\sigma_t^2}{2}x^2\right)\right).
\end{multline}
The function $v_t$ is the unique solution of the Cauchy problem:
\begin{align}
	\label{ODE4} \left\{\begin{array}{ll}
        xv_t''(x)+Q_1(x;t)v_t'(x)+Q_0(x;t) v_t(x)=xP_2(x;t),\cr
				v_t(0)=0 \text{ and } v_t'(0)=0,
    \end{array}
\right.
\end{align}
with:
\begin{align*}
	Q_1(x;t)&=\alpha+\beta-\frac{\beta-\alpha}{\alpha+\beta}x+\frac{(\alpha+\beta)^2-t(\beta-\alpha)^2}{2(\alpha+\beta)^2(1+\alpha+\beta)}x^2,\\
Q_0(x;t)
&=\frac{1}{16(\alpha+\beta)^4(1+\alpha+\beta)^2}\,xP_2(x;t).
\end{align*}
Note that $u_t$ and $v_t$ have the same sign hence proving that $u_t$ is positive (resp. negative) is equivalent to proving that $v_t$ is positive (resp. negative). From standard theory on ODEs \citep{BirkhoffODE,RobinsonODE}, we get that the functions $u_t$ and $v_t$ are $\mathcal{C}^\infty(\mathbb{R})$. Indeed, the only possible singularity is at $x=0$ but the initial conditions imply that the function is regular at this point. In particular, a Taylor expansion at $x=0$ shows that:
\beq \label{VSecond}v_t''(0)=\frac{P_2(0;t)}{1+Q_1(0;t)}=4(\beta^2-\alpha^2)^2(1+\alpha+\beta)(1-t).\eeq
We also observe that the discriminant of the polynomial $x\mapsto P_2(x;t)$ is given by:
\beqq \Delta_t=16t(1+\alpha+\beta)^2(\beta^2-\alpha^2)^2\left((\alpha+\beta)^2-t(\beta-\alpha)^2\right)^2.\eeqq
Hence we conclude that for $t>0$, $P_2$ admits two distinct real zeros that are positive, while for $t< 0$ it remains strictly positive on $\mathbb{R}$. For $t=0$, $P_2$ admits a double zero and thus remains positive on $\mathbb{R}$ appart from its zero.

\medskip

By definition \eqref{vt}, we want to study the sign of $v_t$ on $\mathbb{R}$. Indeed, showing that $v_t$ is positive on $\mathbb{R}$ then $X$ is equivalent to showing $\sigma_t^2$-sub-Gaussianity. We first observe that we may restrict the sign study on $\mathbb{R}_+$. Indeed, if we prove that: 
\beq \label{pass}\forall \,\lambda\geq 0\,\,,\,\, \E[\exp(\lambda X)]=\,_1F_1(\alpha;\alpha+\beta;\lambda)\leq\exp\left(\frac{\lambda \alpha}{\alpha+\beta}+\frac{\lambda^2 \sigma_t^2}{2}\right).\eeq
then, the case $\lambda<0$ is automatically obtained by noting that $1-X\sim\Beta(\beta,\alpha)$, whose mean is $\frac{\beta}{\alpha+\beta} = 1-\frac{\alpha}{\alpha+\beta}$. Therefore, applying \eqref{pass} to $1-X$ gives that for all $\lambda<0$:
\begin{multline*}
	\E[\exp(\lambda X)]=\exp(\lambda)\E[\exp(-\lambda (1-X))]\\
\leq \exp(\lambda)\exp\left(-\lambda (1-\frac{\alpha}{\alpha+\beta})+\frac{\sigma_t^2\lambda^2}{2}\right)=\exp\left(\lambda \frac{\alpha}{\alpha+\beta}+\frac{\sigma_t^2\lambda^2}{2}\right).
\end{multline*}
Eventually, in agreement with the general theory, we observe that for $t>1$ (i.e. $\sigma^2_t<\text{Var}[X]$), $X$ is not $\sigma_t^2$-sub-Gaussian. Indeed, the series expansion at $x=0$ \eqref{VSecond}, shows that for $t>1$, $v_t$ is strictly negative in a neighborhood of $0$. On the contrary, for $t<1$, the function $v_t$ is strictly positive in a neighborhood of $0$ so that we may not directly conclude. Note also that for any value of $t$, we always have $\underset{x\to \infty}{\lim}v_t(x)=+\infty$. 

\subsubsection{Proof that the $\text{Beta}(\alpha,\beta)$ distribution is $\sigma_0^2$-sub-Gaussian\label{SecSigma0}}
In this section, we take $t=0$. As explained above, this corresponds to a case where $P_2$ is positive on $\mathbb{R}$  (apart from its double zero). We prove that $u_0(x)>0$ for $x>0$ by proceeding by contradiction. Let us assume that there exists $x_1>0$ such that $u_0(x_1)=0$. Since the non-empty set $\{x>0 \,/\, u_0(x)=0\}$ is compact (because $u_0\in \mathcal{C}^\infty(\mathbb{R})$) and excludes a neighborhood of $0$, we may define $x_0=\min\{x>0 \text{ such that } u_0(x)=0\}>0$. Let us now define the set:
\beqq M=\{0<x<x_0 \text{ such that } u_0'(x)=0 \text{ and } u_0' \text{ changes sign at } x\}.\eeqq
Since $u_0(0)=u_0(x_0)=0$ and the facts that $u_0$ is strictly positive in a neighborhood of $0$ and $u_0'$ is continuous on $\mathbb{R}$, Rolle's theorem shows that $M$ is not empty and that:
\beqq m=\min\{x\in M\}\text{ exists and } 0<m<x_0.\eeqq
Evaluating the ODE \eqref{ODE3} at $x=m$ and using the fact that the polynomial $P_2$ is positive on $\mathbb{R}$ (appart from its double zero) leads to:
\beq\label{eq:FinalClue} mu_0''(m)+0-\alpha u_0(m)\geq 0 \,\Rightarrow \, u_0''(m)\geq \frac{\alpha}{m}u_0(m)>0.\eeq
However, combined with $u_0'(m)=0$, this contradicts the fact that $u_0'$ changes sign at $x=m$. 

\medskip

Thus, we conclude that there cannot exist $x_1>0$ such that $u_0(x_1)=0$. Since $u_0$ is strictly positive in a neighborhood of $0$ and continuous on $\mathbb{R}$, we conclude that it must remain strictly positive on $\mathbb{R}_+^*$.

\begin{rem} In this proof, the case $\beta=\alpha$ requires an adaptation since $u_0''(0)=0$. Thus, we must determine $u_0^{(4)}(0)>0$ ($u_0^{(3)}(0)=0$ by symmetry) to ensure that the function $u_0$ is locally convex and remains strictly positive in a neighborhood of $0$. Apart from this minor verification, the rest of the proof applies also to this case.  
\end{rem}

\subsubsection{Proof of the optimal proxy variance for the $\text{Beta}(\alpha,\beta)$ distribution\label{FIG2}}   
In this section we assume that $\beta\neq \alpha$. From general theorems regarding ODEs, we have that the application:
\beq \label{DoubleMap} 
g :\left\{\begin{array}{cll}
      [0,+\infty)\times [0,+\infty) & \to & \mathbb{R} \\
      (x,t) & \mapsto &  g(x,t)=v_t(x),\\
\end{array} 
\right.
\eeq
is smooth ($g\in \mathcal{C}^\infty([0,+\infty)\times [0,+\infty))$). Indeed, the $t$-dependence of the coefficients of the ODE \eqref{ODE4} is polynomial and thus smooth. The $x$-dependence of the coefficients of the ODE \eqref{ODE4} is polynomial and as explained above, the only possible singularity in $x$ is at $x=0$ but initial conditions ensure that the solutions $x\mapsto v_t(x)$ are always regular there. Since for all $t\geq 0$ we have $\underset{x\to \infty}{\lim}v_t(x)=+\infty$, we also have that the function:
\beqq h :\left\{\begin{array}{cll}
      [0,+\infty) & \to & \mathbb{R} \\
      t & \mapsto &  \min\{v_t(x),\,x\in \mathbb{R}_+^*\},\\
\end{array} 
\right.
\eeqq
is continuous.

\medskip

We now observe that for any $0\leq t< 1$, the functions $v_t$ are strictly positive in a neighborhood of $0$. More precisely, if we choose a segment $[0,t_0]$ with $t_0<1$, then for all $0\leq t\leq t_0$ we have that $v_t''(0)\geq v_{t_0}''(0)=4(\beta^2-\alpha^2)^2(1+\alpha+\beta)(1-t_0)>0$. Hence, we may choose $\eta>0$ such that for all $0\leq t\leq t_0$, we have $v_t$ is strictly positive on $]0,\eta]$. Moreover, since $\underset{x\to\infty }{\lim} v_0(x)=+\infty$, $v_0$ is bounded from below on $[\eta,+\infty[$ by a constant $A>0$. Thus, since $g$ is continuous, there exists a neighborhood of $t=0$ in which all solutions $v_t$ remain greater than $\frac{A}{2}$ on $[\eta,+\infty[$ and thus strictly positive on $\mathbb{R}_+^*$. This shows that for $t>0$, $\sigma_0^2$ is not optimal.

\medskip

Let us now introduce the set:
\beqq T_+=\{t\geq 0 \text{ such that } v_t \text{ is positive on } \mathbb{R}_+^*\}. \eeqq
Then, from the results presented above, we know that $T_+$ is non-empty, that it contains a neighborhood of $0$ and that it is bounded from above by $1$. Moreover, by connection with the initial problem \eqref{vt}, $T_+$ is an interval and thus is of the form $[0,t_{\text{opt}}]$ with $0<t_{\text{opt}}< 1$. Indeed $t\in T_+$ implies by construction that for all $s\leq t$, $s\in T_+$. Note also that $t<1$ since $v_1$ is strictly negative in a neighborhood of $0$ and thus $\min\{v_1(x),x\in \mathbb{R}_+^*\}<0$. Hence, the continuity of $h$ shows that there exists a neighborhood of $t=1$ in which the solutions $v_t$ are non-positive on $\mathbb{R}_+^*$. 
For $t=t_{\text{opt}}$ the function $v_{t_{\text{opt}}}$ must have a zero on $\mathbb{R}_+^*$ otherwise by continuity of $h$ we may find a neighborhood of $t_{\text{opt}}$ for which $\min\{v_t(x),\,x\in \mathbb{R}_+^*\}$ remains strictly positive thus contradicting the maximality of $t_{\text{opt}}$. Since $v_{t_{\text{opt}}}$ must remain positive, the zero is at least a double zero and therefore we find that there exists $x_0>0$ such that $v_{t_{\text{opt}}}(x_0)=0$, $v'_{t_{\text{opt}}}(x_0)=0$ and $v''_{t_{\text{opt}}}(x_0)\geq 0$. From \eqref{eq:HyperGeo} and \eqref{vt}, the conditions $v_{t_{\text{opt}}}(x_0)=0$, $v'_{t_{\text{opt}}}(x_0)=0$ are equivalent to the following system of equations (we use here the contiguous relations for the confluent hypergeometric function: $_1F_1'(a;b;x)=\frac{a}{b}~_1F_1(a;b;x)$):
\beq \label{SystemEq} \left\{\begin{array}{ll}
        _1F_1(\alpha;\alpha+\beta;x_0)=\exp\left(\frac{\alpha}{\alpha+\beta}x_0+\frac{\sigma^2_{t_{\text{opt}}}}{2}x_0^2\right),\cr
		\frac{\alpha}{\alpha+\beta}~_1F_1(\alpha+1;\alpha+\beta+1;x_0)=\left(\frac{\alpha}{\alpha+\beta}+\sigma^2_{t_{\text{opt}}}x_0\right)~_1F_1(\alpha;\alpha+\beta;x_0). 
    \end{array}
\right.
\eeq      
This is equivalent to say that $x_0\equiv x_0(\alpha,\beta)$ is the solution of the transcendental equation:
\beqq 
_1F_1(\alpha;\alpha+\beta;x_0)=\exp\left(\frac{\alpha x_0}{2(\alpha+\beta)}\left(1+\frac{_1F_1(\alpha+1;\alpha+\beta+1;x_0)}{_1F_1(\alpha;\alpha+\beta;x_0)}\right)\right),\eeqq
and that $\sigma^2_{t_{\text{opt}}}$ is given by:
\beqq
\sigma^2_{t_{\text{opt}}}=\frac{\alpha}{(\alpha+\beta)x_0}\left( \frac{_1F_1(\alpha+1;\alpha+\beta+1;x_0)}{_1F_1(\alpha;\alpha+\beta;x_0)}  -1\right). 
\eeqq
Note that by symmetry, we have $x_0(\beta,\alpha)=-x_0(\alpha,\beta)$ hence, $\sigma^2_{t_{\text{opt}}}(\beta,\alpha)=\sigma^2_{t_{\text{opt}}}(\alpha,\beta)$. Moreover, if $\beta>\alpha$ then $x_0(\alpha,\beta)>0$ while $\alpha>\beta$ implies $x_0(\alpha,\beta)<0$. 
We may illustrate the situation with Figure~\ref{fig:proof} which displays the difference function $x\mapsto u_t(x)$. 
\begin{rem} The system of equations \eqref{SystemEq} admits only one solution on $\mathbb{R}_+^*$. Indeed, let us transpose the problem from $v_{t_{\text{opt}}}$ to $u_{t_{\text{opt}}}$ using \eqref{vt} and assume that there exist two points $0<x_0<x_1$ such that $u_{t_{\text{opt}}}(x_0)=0$, $u'_{t_{\text{opt}}}(x_0)=0$ and $u_{t_{\text{opt}}}(x_1)=0$, $u'_{t_{\text{opt}}}(x_1)=0$ with $u_{t_{\text{opt}}}$ strictly positive on $(x_0,x_1)$ (hence $u''_{t_{\text{opt}}}(x_0)\geq 0$ and $u''_{t_{\text{opt}}}(x_1)\geq 0$). Using \eqref{ODE3}, this implies that $P_2(x_0;t_{\text{opt}})\geq 0$ and $P_2(x_1;t_{\text{opt}})\geq 0$. If we denote $x_-<x_+$ the potential distinct positive zeros of $x\mapsto P_2(x;t_{\text{opt}})$ we may exclude that $x_0\leq x_-$. Indeed, if $x_0\leq x_-$ then we may apply the same argument to $u_{t_{\text{opt}}}$ on the interval $[0,x_0]$ as the one developed for $u_0$ in Section \ref{SecSigma0} and obtain a contradiction. Thus, the only remaining case is to assume $x_1>x_0>x_+$. In that case, since $u_{t_{\text{opt}}}(x_0)=0$, $u'_{t_{\text{opt}}}(x_0)=0$, $u_{t_{\text{opt}}}(x_1)=0$ and $x\mapsto P_2(x;t_{\text{opt}})$ is positive on $[x_0,x_1]$, we may apply the same argument to $u_{t_{\text{opt}}}$ on the interval $[x_0,x_1]$ as the one developed for $u_0$ in Section \ref{SecSigma0} and obtain a contradiction. 
\end{rem}
\begin{figure}
\begin{center}
\includegraphics[height=.4\textwidth]{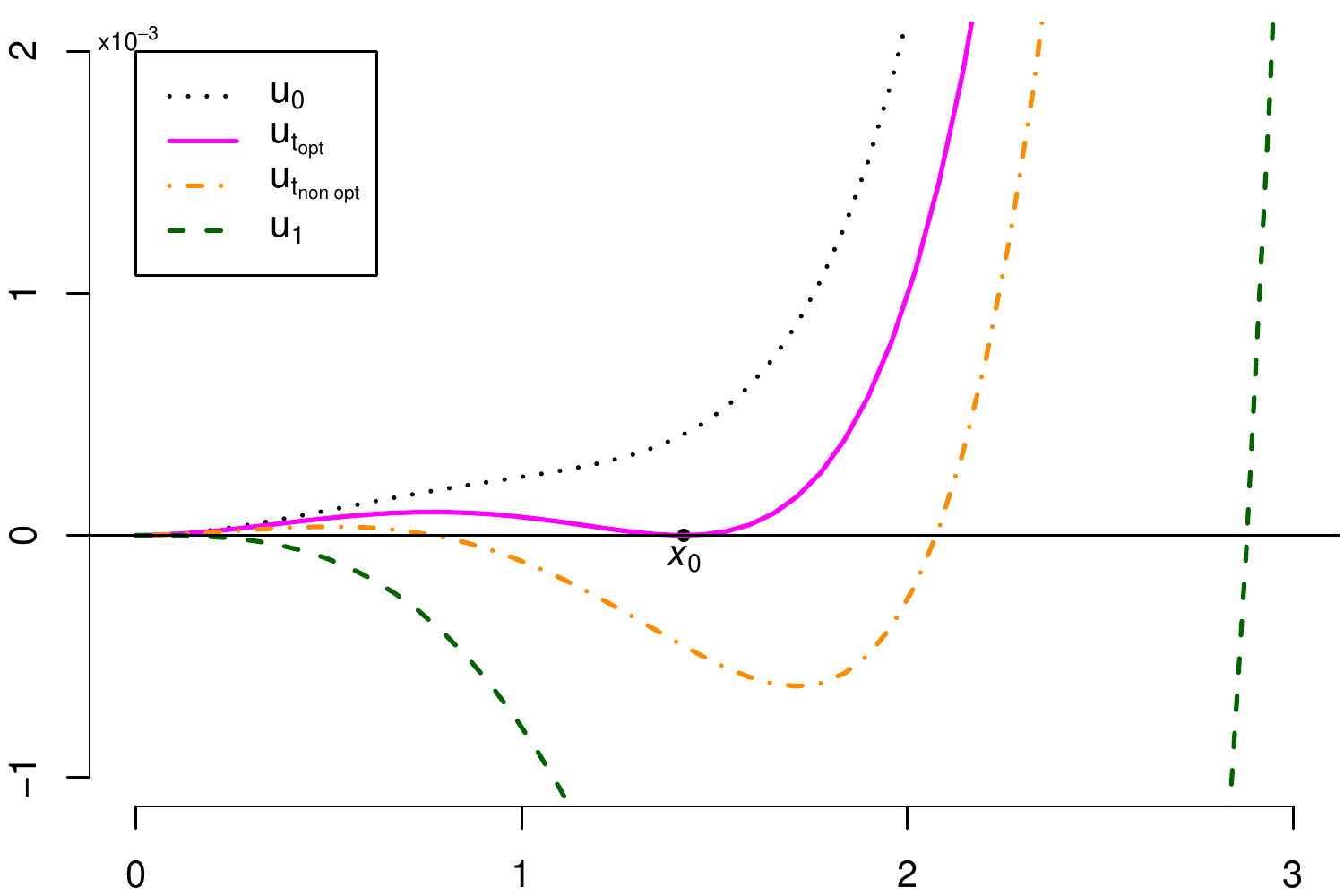}
\end{center}
\caption{\label{fig:proof}
Difference function $x\mapsto u_t(x)$. For $t=0$ (simple upper bound $\sigma_0^2$), the curve [dotted black] remains strictly positive. For $t=t_{\text{opt}}$ (optimal proxy variance $\sigma^2_{\text{opt}}$), the curve [magenta] has zero minimum (at $x_0$). For $t=1$ (leading to the variance), the curve [dashed green] has negative second derivative at $x=0$, hence is directly negative around 0. The intermediate case with $t_{\text{non opt}}$ in the interval $(t_{\text{opt}},1)$ produces a curve [orange, dash and dots] which is first positive, then negative, and positive again.
}
\end{figure}

\section{Relations to other distributions\label{sec:dir}}

\subsection{Optimal proxy variance for the Bernoulli distribution}

We show that our proof technique can be used to recover the optimal proxy variance for the Bernoulli distribution, known since \cite{kearns1998large}. This is illustrated by the center panel of Figure~\ref{fig:proxy}.

\begin{thm}[Optimal proxy variance for the Bernoulli distribution]\label{thm:3}
For any $\mu\in(0,1)$, the Bernoulli distribution with mean $\mu$ is sub-Gaussian with optimal proxy variance $\sigma_{\text{opt}}^2(\mu)$ given by:
\begin{equation}\label{eq:def_sig_mu}
        \sigma_{\text{opt}}^2(\mu)= \frac{(1-2\mu)}{2\ln \frac{1-\mu}{\mu}}.
\end{equation}
\end{thm}

\begin{proof}
In the limit $\alpha\to 0$ with $\frac{\alpha}{\alpha+\beta}$ fixed equal to $\mu$, the differential equation~\eqref{ODE3} simplifies into: 
\begin{multline*} 
u_{t,\mu}''(x)-u_{t,\mu}(x) = \exp\left(\mu x+\frac{x^2}{8}-\frac{x^2(2\mu-1)^2t}{8}\right)\\
\left(\frac{(2\mu-1)^2(1-t)}{4}+\frac{(2\mu-1)(1-t+4t\mu(1-\mu))}{4}x-\frac{(1-t+4t\mu(1-\mu))^2}{16}x^2\right)
\end{multline*}
with the Cauchy initial conditions $u_{t,\mu}(0)=0$ and $u'_{t,\mu}(0)=0$. The solution of this Cauchy problem is explicit and given by:
\begin{equation} \label{SolutionAlpha0} u_{t,\mu}(x)= \exp\left(\mu x+\frac{x^2}{8}-\frac{x^2(2\mu-1)^2t}{8}\right) -\mu \edr^{x}+\mu-1
\end{equation}
Therefore the optimal proxy variance is given by $\sigma_{\text{opt}}^2(\mu)=\frac{1}{4}-\frac{1}{4}x(2\mu-1)^2t_0$ where $t_0$ is determined by the system of equations: $u_{t_0,\mu}(x_0)=0$ and $u'_{t_0,\mu}(x_0)$, thus defining implicitly $t_0$ and $x_0$ as functions of $\mu$. In order to solve explicitly the last system of equations, we perform the change of variables: $(\mu,t)=\left(\frac{s}{s+1},\frac{2\td{t}}{s+1}+1\right)$ so that the solution \eqref{SolutionAlpha0} is now given by:
\begin{equation*} 
u_{\td{t},s}(x)=\exp\left(\frac{s}{s+1}x-\frac{\td{t}}{4(s+1)}x^2\right)-\frac{s}{s+1}\edr^{x}-\frac{1}{s+1}
\end{equation*}
Consequently, we have to solve the system:
\begin{equation*}
\left\{
\begin{array}{l}
 u_{\td{t}_0,s}(x_0)=0 \\  
 u'_{\td{t}_0,x}(x_0)=0
\end{array}
\right.
\,\,\Leftrightarrow \,\, \left\{\begin{array}{l}
 (s+1)\exp\left(\frac{s}{s+1}x_0-\frac{\td{t}_0}{4(s+1)}x_0^2\right)=s\edr^{x_0}+1 \\  
 \left(s-\frac{x_0\td{t}_0}{2}\right)\exp\left(\frac{s}{s+1}x_0-\frac{\td{t}_0}{4(s+1)}x_0^2\right)=s\edr^{x_0}
\end{array}
\right.
\end{equation*}
Introducing another change of variable $(x_0,y_0)=(x_0,x_0\td{t}_0)$, the last system is equivalent to:
\begin{equation*} 
\left\{\begin{array}{l}
 s-y_0=s(1+y_0)\edr^{x_0} \\  
 1+y_0=\exp\left(-\frac{x_0}{s+1}\left(s-\frac{y_0}{2}\right)\right)
\end{array}
\right. \,\,\,\Leftrightarrow  \left\{\begin{array}{l}
 s-y_0=s\,\exp\left(\frac{y_0+2}{y_0-2s}\ln(1+y_0)\right) \\  
 x_0=\frac{s+1}{\frac{y_0}{2}-s}\ln(1+y_0)
\end{array}
\right.
\end{equation*}
We now observe that $y_0=s-1$ and $x_0=-2\ln s$ is a solution of the former system. Performing back the various changes of variables, this is equivalent to say that $\td{t}_0=\frac{y_0}{x_0}= \frac{2(1-s)}{\ln s}$ so that $t_0=\frac{4(1-s)}{(s+1)\ln s}+1$ or equivalently $t_0=\frac{1}{(2\mu-1)^2}+\frac{2}{(2\mu-1)\ln\frac{1-\mu}{\mu}}$. Consequently, the optimal proxy variance is given by:
\begin{equation*} 
\sigma_{\text{opt}}^2(\mu)=\frac{1}{4}-\frac{1}{4}(2\mu-1)^2t_0=\frac{(1-2\mu)}{2\ln \frac{1-\mu}{\mu}}
\end{equation*}
which is precisely the optimal proxy variance of a Bernoulli random variable with mean $\mu$.
\end{proof}

\subsection{Optimal proxy variance for the Dirichlet distribution}

We start by reminding the definition of sub-Gaussian property for random vectors:

\begin{defi}[Sub-Gaussian vectors]\label{Def2}
A random $d$-dimensional vector $\boldsymbol{X}$ with finite  mean $\bdmu=\E[\boldsymbol{X}]$ is $\sigma^2${-sub-Gaussian} if the random variable $\boldsymbol{u}^\top \,\boldsymbol{X}$ is $\sigma^2$-{sub-Gaussian} for any unit vector $\boldsymbol{u}$ in the simplex $\mathcal{S}^{d-1}=\{\boldsymbol{u}\in [0,1]^d \,/\,\underset{i=1}{\overset{d}{\sum}}u_i=1\}$. This is equivalent to say that: 
\begin{align*}
\E[\exp(\bdlambda^\top (\boldsymbol{X}-\bdmu))]\le\exp\left(\frac{\| \bdlambda\|^2\sigma^2}{2}\right)\,\,\text{for all } \bdlambda\in\R^d.
\end{align*}
where $\| \bdlambda\|^2=\underset{i=1}{\overset{d}{\sum}} \lambda_i^2$. Eventually, a random vector $\boldsymbol{X}$ is said to be strictly sub-Gaussian, if the random variables $\boldsymbol{u}^\top \,\boldsymbol{X}$ are strictly sub-Gaussian for any unit vector $\boldsymbol{u}\in \mathcal{S}^{d-1}$.
\end{defi}
%

Let $d\geq 2$. The Dirichlet distribution $\Dir(\bdalpha)$, with positive parameters $\bdalpha=(\alpha_1,\ldots,\alpha_d)^\top$,  is characterized by a density on the simplex $\mathcal{S}^{d-1}$ given by:
\beqq 
f(x_1,\ldots,x_d;\alpha_1,\dots,\alpha_d) = \frac{1}{B(\bdalpha)} \prod_{i=1}^dx_i^{\alpha_i-1},
\eeqq
where $B(\boldsymbol{\alpha})=\frac{1}{\Gamma(\bar{\bdalpha})}\underset{i=1}{\overset{d}{\prod}} \Gamma(\alpha_i)$ and $\bar{\bdalpha}=\underset{i=1}{\overset{d}{\sum}}\alpha_i$.
It generalizes the Beta distribution in the sense that the components are $\text{Beta}$ distributed. More precisely, for any non-empty and strict subset $\mathcal{I}$ of $\{1,\ldots,d\}$:
\beqq \boldsymbol{X}=(X_1,\dots,X_d)^\top\sim \Dir(\bdalpha) \,\Longrightarrow\, \underset{i\in\mathcal{I}}\sum X_i\sim\text{Beta}\left(\underset{i\in\mathcal{I}}\sum \alpha_i,\underset{j\notin\mathcal{I}}\sum\alpha_j\right).
\eeqq
However, we remind the reader that the components $(X_i)_{1\leq i\leq d}$ are not independent and the variance/covariance matrix is given by:
\beqq \forall \,i\neq j\,: \, \text{Cov}[X_i,X_j]=-\frac{\alpha_i\alpha_j}{\bar{\bdalpha}^2(1+\bar{\bdalpha})}\,\text{ and }\,\forall\, 1\leq i\leq d\,\,,\,\, \text{Var}[X_i]=\frac{\alpha_i(\bar{\bdalpha}-\alpha_i)}{\bar{\bdalpha}^2(1+\bar{\bdalpha})}.
\eeqq
Eventually, if we define $\bdn=(n_1,\dots,n_d)^\top\in \mathbb{N}^d$, then the moments of the Dirichlet distribution are given by:
\beqq \E\left[\prod_{i=1}^d X_i^{n_i}\right]=\frac{B(\bdalpha+\bdn)}{B(\bdalpha)}.
\eeqq
This is equivalent to say that the moment-generating function of the Dirichlet distribution is:
\begin{align*}
	\E[\text{exp}(\bdlambda^\top\boldsymbol{X})]&=\sum_{m=0}^\infty\,\sum_{n_1+\dots+n_d=m}\frac{B(\bdalpha+\bdn)}{B(\bdalpha) }\frac{\lambda_1^{n_1}\dots \lambda_d^{n_d}}{(n_1)!\dots (n_d)!},\\
&=\sum_{m=0}^\infty\,\sum_{n_1+\dots+n_d=m}\frac{\lambda_1^{n_1}\dots \lambda_d^{n_d}}{(n_1)!\dots (n_d)!}\frac{\Gamma(\bar{\bdalpha})}{\Gamma(\bar{\bdalpha}+\bar{\bdn})}\prod_{i=1}^d \frac{\Gamma(\alpha_i+n_i)}{\Gamma(\alpha_i)}, \\
&=\sum_{m=0}^\infty\,\sum_{n_1+\dots+n_d=m}\frac{\lambda_1^{n_1}\dots \lambda_d^{n_d}}{(n_1)!\dots (n_d)!  }\frac{1}{\left(\bar{\bdalpha}\right)_{\bar{\bdn}}} \prod_{i=1}^d \left(\alpha_i\right)_{n_i},
\end{align*}
where we have defined $\bdlambda=(\lambda_1,\dots,\lambda_d)^\top\in \mathbb{R}^d$ and $\bar{\bdn}=\underset{i=1}{\overset{d}{\sum}}n_i.$

Let us define $\boldsymbol{e}_i$ the $i^{\text{th}}$ canonical vector of $\mathbb{R}^d$ and $\boldsymbol{X}=(X_1,\dots,X_d)^\top\sim\Dir(\bdalpha)$. From Definition~\ref{Def1} and the results regarding the $\text{Beta}(\alpha,\beta)$ distribution obtained in Section \ref{sec:distinct}, we immediately get that $\boldsymbol{e}_i^\top \, \boldsymbol{X}=X_i$ is $\sigma_i^2$-sub-Gaussian with $\sigma_i^2\overset{\text{def}}{=}\sigma_{\text{opt}}^2(\alpha_i,\bar{\bdalpha}-\alpha_i)$ defined from Theorem \ref{thm:1}. Moreover, in direction $\boldsymbol{e}_i$, $\sigma_i^2$ is the optimal proxy variance. Therefore, the remaining issue is to generalize these results for arbitrary unit vectors on $\mathcal{S}^{d-1}$. We obtain the following result:

\begin{thm}[Optimal proxy variance for the Dirichlet distribution]\label{thm:2}
For any parameter $\bdalpha$, the Dirichlet distribution $\Dir(\bdalpha)$ is sub-Gaussian with optimal proxy variance $\sigma_{\text{opt}}^2(\bdalpha)$ given from Theorem~\ref{thm:1} and:
\beqq \sigma_{\text{opt}}^2(\bdalpha)=\sigma_{\text{opt}}^2(\alpha_{\text{max}},\bar{\bdalpha}-\alpha_{\text{max}})
\quad\text{ where } \quad \alpha_{\text{max}} = \max_{1\leq i\leq d}\{\alpha_i\}.
\eeqq
\end{thm}

\begin{proof}
We first observe that the computations of $\sigma^2_{\text{opt}}(\alpha_i,\beta_i=\bar{\bdalpha}-\alpha_i)$ correspond to cases where the sum $\alpha_i+\beta_i$ is fixed to $\bar{\bdalpha}$ and thus independent of $i$. Therefore, 
$\sigma^2_{\text{opt}}(\alpha_i,\bar{\bdalpha}-\alpha_i)$ is maximal when $|(\bar{\bdalpha}-\alpha_i)-\alpha_i|$ is minimal, i.e. when the distance from $\alpha_i$ to $\frac{1}{2}\bar{\bdalpha}$ is minimal. It is easy to see that this corresponds to choosing $\alpha_i = \alpha_{\text{max}} = \max\{\alpha_i\,\,,\,\, 1\leq i\leq d\}$  (by looking at the two possible cases $\alpha_{\text{max}}\leq \frac{1}{2}\bar{\bdalpha}$ and $\alpha_{\text{max}} > \frac{1}{2}\bar{\bdalpha}$).

We then observe that $\sigma^2_{\text{max}}$ cannot be improved. Indeed, let us denote $i_0$ one of the components for which the maximum is obtained. Then, if we take $\boldsymbol{u}=\boldsymbol{e}_{i_0}$, the discussion presented above shows that $\sigma_{i_0}^2=\sigma^2_{\text{max}}$ is the optimal proxy variance in this direction. Hence the optimal proxy variance cannot be lower than $\sigma^2_{\text{max}}$.\newline
Let us now prove that $\boldsymbol{X}$ is $\sigma^2_{\text{max}}$-sub-Gaussian. Let $\boldsymbol{u}=(u_1,\dots,u_d)^\top$ be a unit vector on $\mathcal{S}^{d-1}$ and $\lambda\in \mathbb{R}$. We define for clarity $\bdlambda=\lambda \boldsymbol{u}$. We have:
\begin{multline}
	\label{ExpectationProduct}\E\left[\exp\left(\lambda \boldsymbol{u}^\top \boldsymbol{X}\right)\right]=\E\left[\exp\left(\bdlambda^\top \boldsymbol{X}\right)\right]=\E\left[\exp\left(\sum_{i=1}^d\lambda_i X_i\right)\right],\\
=\sum_{m=0}^\infty\,\sum_{n_1+\dots+n_d=m}\frac{\lambda_1^{n_1}\dots \lambda_d^{n_d}}{(n_1)!\dots (n_d)!  }\frac{1}{\left(\bar{\bdalpha}\right)_{\bar{\bdn}}} \prod_{i=1}^d \left(\alpha_i\right)_{n_i}.
\end{multline}
Note that we also have:
\begin{align}
	\prod_{i=1}^d \E\left[\exp(\lambda_i X_i)\right]&=\prod_{i=1}^d \left(\sum_{j=0}^\infty \frac{(\alpha_i)_j}{(j!) (\bar{\bdalpha})_j}\lambda_i^j\right)\nonumber\\
&=\sum_{m=0}^\infty\,\sum_{n_1+\dots+n_d=m}\frac{\lambda_1^{n_1}\dots \lambda_d^{n_d}}{(n_1)!\dots (n_d)!  }\prod_{i=1}^d\frac{(\alpha_i)_{n_i}}{(\bar{\bdalpha})_{n_i}}.\label{ProductExpectation} 
\end{align}
Moreover we have the inequality:
\beqq 
\prod_{i=1}^d (\bar{\bdalpha})_{n_i}\leq (\bar{\bdalpha})_{\bar{\bdn}}, \eeqq
because both sides have the same number of terms in the product (i.e. $\bar{\bdn}$) but those of the right hand side are always greater or equal to those of the left hand side. Hence, from \eqref{ExpectationProduct} and \eqref{ProductExpectation}, we find:
\begin{equation*}
\E\left[\exp\left(\bdlambda^\top (\boldsymbol{X}-\bdmu)\right)\right]\leq \prod_{i=1}^d \E\left[\exp(\lambda_i (X_i-\mu_i))\right].
\end{equation*}
Using the optimal proxy variance of the Beta distribution proven in Theorem~\ref{thm:1}, we find:
\begin{multline*}
	\E\left[\exp\left(\bdlambda^\top (\boldsymbol{X}-\bdmu)\right)\right]\leq\prod_{i=1}^d \E\left[\exp(\lambda_i (X_i-\mu_i))\right]\leq\prod_{i=1}^d \exp\left(\frac{\lambda_i^2\sigma_i^2}{2}\right)\\
\leq\prod_{i=1}^d \exp\left(\frac{\lambda_i^2\sigma^2_{\text{max}}}{2}\right)=\exp\left(\frac{\sigma^2_{\text{max}} \|\bdlambda\|^2}{2}\right),
\end{multline*}
thus showing that $\boldsymbol{X}$ is $\sigma^2_{\text{max}}$-sub-Gaussian and concluding the proof.
\end{proof}

Note that using Theorem \ref{thm:2}, we obtain the following corollary:
\begin{cor}\label{cor2}
For any integer $d\geq 2$, the Dirichlet distribution $\Dir(\alpha_1,\dots,\alpha_d)$ is strictly sub-Gaussian if and only if $d=2$ and $\alpha_1=\alpha_2$.
\end{cor}

Indeed, we first need to require $\alpha_1=\dots=\alpha_d\overset{\text{def}}{=}\alpha$ so that all directions have the same optimal proxy variance. Then, each component satisfies $X_i=\boldsymbol{e}_i^\top \boldsymbol{X} \sim \text{Beta}(\alpha, (d-1)\alpha)$ and Theorem~\ref{thm:1} shows that $\sigma_i^2$ is the optimal proxy variance for $X_i$ if and only if $\alpha=(d-1)\alpha$, i.e. if and only if $d=2$.

\section*{Acknowledgments}

We wish to thank St\'ephane Boucheron  for an enlightening discussion of our results and an anonymous referee for insightful suggestions.  
O.M. would like to thank Universit\'e de Lyon, Universit\'e Jean Monnet and Institut Camille Jordan for financial support. This work was supported by the LABEX MILYON (ANR-10-LABX-0070) of Universit\'e de Lyon, within the program ``Investissements d'Avenir'' (ANR-11-IDEX-0007) operated by the French National Research Agency (ANR).
J.A. would like to thank Inria Grenoble Rh\^one-Alpes and 
Laboratoire Jean Kuntzmann, Universit\'e Grenoble Alpes for financial support.  J.A. is also member of Laboratoire de Statistique, CREST, Paris. This work was partially conducted during a scholar visit of J.A. at the Department of Statistics \& Data Science of the University of Texas at Austin.

\bibliographystyle{apalike}
\bibliography{biblio}

\end{document}